\documentclass[12pt]{article}

\usepackage[margin=1in]{geometry}
\usepackage{graphicx}             
\usepackage{amsmath}               
\usepackage{amsfonts}              
\usepackage{amsthm}                
\usepackage{amssymb}
\usepackage{amscd}
\usepackage{extarrows}
\usepackage[shortlabels]{enumitem}
\usepackage{secdot}
\usepackage{tikz}
\usepackage{sectsty}

\sectionfont{\center \normalsize}
\subsectionfont{\center \normalsize}
\subsubsectionfont{\normalsize}
\paragraphfont{\normalsize}

\sectiondot{subsection}

\newtheorem{thm}{Theorem}[section]

\newtheorem{lem}[thm]{Lemma}
\newtheorem{prop}[thm]{Proposition}
\newtheorem{cor}[thm]{Corollary}

\theoremstyle{definition}

\theoremstyle{remark}

\theoremstyle{remark}

\theoremstyle{remark}
\newtheorem*{sq}{\bf Serre's Question}

\theoremstyle{remark}
\newtheorem*{tq}{\bf Totaro's Question}

\theoremstyle{remark}

\theoremstyle{remark}

\DeclareMathOperator{\Aut}{Aut}
\DeclareMathOperator{\Br}{Br}
\DeclareMathOperator{\id}{id}
\DeclareMathOperator{\ind}{ind}
\DeclareMathOperator{\per}{per}

\DeclareMathOperator{\ch}{char}
\DeclareMathOperator{\Gal}{Gal}
\DeclareMathOperator{\tr}{tr}
\DeclareMathOperator{\Spec}{Spec}

\DeclareMathOperator{\Hom}{Hom}

\DeclareMathOperator{\PGL}{PGL}
\DeclareMathOperator{\SO}{SO}
\DeclareMathOperator{\chargroup}{\textbf{X}}

\DeclareMathOperator{\GL}{GL}
\DeclareMathOperator{\Z}{\mathbb{Z}}
\DeclareMathOperator{\Q}{\mathbb{Q}}
\DeclareMathOperator{\G}{\mathbb{G}}

\makeatletter
\newcommand{\etale}{\'etal\@ifstar{\'e}{e\space}}
\makeatother
\makeatletter
\newcommand{\CT}{Colliot-Th\'el\`en\@ifstar{\'e}{e}}
\makeatother
\makeatletter
\newcommand{\Vosk}{Voskresenski{\u\i}}
\makeatother

\newcommand*{\TitleFont}{%
      \usefont{\encodingdefault}{\rmdefault}{b}{n}%
      \fontsize{14}{20}%
      \selectfont}

\newcommand*{\AddFont}{%
      \usefont{\encodingdefault}{\rmdefault}{}{n}%
      \fontsize{10}{20}%
      \selectfont}

\begin{document}

\nocite{*}

\title{\TitleFont \textbf{TOTARO'S QUESTION FOR TORI OF LOW RANK}}
\author{REED LEON GORDON-SARNEY}
\date{\AddFont DEPARTMENT OF MATHEMATICS \& COMPUTER SCIENCE \\ EMORY UNIVERSITY, ATLANTA, GA 30322 USA}

\maketitle

\begin{abstract}
Let $G$ be a smooth connected linear algebraic group and $X$ be a $G$-torsor.
Totaro asked: if $X$ admits a zero-cycle of degree $d \geq 1$, then does $X$ have a closed \etale point of degree dividing $d$?
This question is entirely unexplored in the literature for algebraic tori.
We settle Totaro's question affirmatively for algebraic tori of rank $\leq 2$.
\end{abstract}
\vspace{0.15cm}
\section{Introduction}
Let $X$ be a smooth quasiprojective variety over a field $k$.
Define its \emph{index}, denoted $\ind(X)$, to be the minimal positive degree of a zero-cycle on $X$.
This is nothing more the greatest common divisor of degrees of field extensions $L/k$ such that $X(L) \neq \varnothing$.
If $X$ has a rational point, then clearly $\ind(X) = 1$; but the converse is false in general.
Striking counterexamples to the converse are found among conic bundles over $\mathbb{P}_{\Q_p}^1$ (due to \CT--Coray \cite{colliocoray79}), affine homogeneous spaces under a smooth connected linear algebraic group over $\Q_p$ with finite stabilizers (due to Florence \cite{florence04}), and projective homogeneous spaces under a smooth connected linear algebraic group over $\Q_p((t))$ (due to Parimala \cite{pari05}).

Serre asked if every index 1 \emph{principal} homogeneous space (or torsor) under a smooth connected linear algebraic group $G$ over a field $k$ has a rational point \cite{serre95}.
Such spaces are classified by the pointed Galois cohomology set $H^1(k,G)$; for any $X \in H^1(k,G)$ and any field extension $L/k$, $X(L) \neq \varnothing$ if and only if $X_L = 1 \in H^1(L,G_L)$.
So the index of a $G$-torsor $X$ over $k$ is exactly the greatest common divisor of degrees of field extensions $L/k$ such that $X_L = 1 \in H^1(L,G_L)$.
Rephrased in the language of Galois cohomology,
\begin{sq}[1995]
Let $G$ be a smooth connected linear algebraic group over a field $k$, and let $X \in H^1(k,G)$ be a $G$-torsor over $k$. 
If $\ind(X) = 1$, then is $X = 1 \in H^1(k,G)$?
\end{sq}
No counterexamples to Serre's question are known, and there are positive answers in some special cases:
the case of $\PGL_n$ is known from the classical theory of central simple algebras;
the case of $\SO_n$ is due to Springer \cite{springer52};
the case of unitary groups is a result of Bayer--Lenstra \cite{bayerlenstra90};
and Sansuc proved that Serre's question has an positive answer for any smooth connected linear algebraic group over a number field or a $p$-adic field \cite{sansuc81}.
One should refer to Black \cite{black11a, black11b} for further work on this question.

However, for abelian $G$, a positive answer to Serre's question is a trivial consequence of the fact that the order of $X$ in the \emph{abelian group} $H^1(k,G)$, called the \emph{period} of $X$ and denoted $\per(X)$, divides $\ind(X)$ (cf. Lemma 3.1).
Totaro generalized Serre's question in a natural way that was non-obvious even for abelian $G$: he asked if the existence of a zero-cycle on $X$ of degree $d \geq 1$ implies the existence of a closed \etale point on $X$ of degree dividing $d$ \cite{totaro04}.
Reformulating in the language of Galois cohomology as before,
\begin{tq}[2004]
Let $G$ be a smooth connected linear algebraic group over a field $k$, and let $X \in H^1(k,G)$ be a $G$-torsor over $k$. 
Is there a separable field extension $F/k$ of degree $\ind(X)$ such that $X_F = 1 \in H^1(F,G_F)$?
\end{tq}
No counterexamples to Totaro's question are known, but affirmative proofs are scarcer than those for Serre's question:
the case of $\PGL_n$ is again a classical theorem about central simple algebras;
in the paper where he first asked the question, Totaro answered it positively for split simply connected groups of type $G_2$, $F_4$, or $E_6$ (with a partial result for $E_7$) \cite{totaro04};
Garibaldi--Hoffman improved upon this result to give an positive answer for groups of type $G_2$, reduced of type $F_4$, and simply connected of types ${}^1 E_{6,6}^{\, 0}$ or ${}^1 E_{6,2}^{\, 28}$ \cite{garihoff06};
and Black--Parimala settled the question for simply connected semisimple groups of rank $\leq 2$ over fields of characteristic $\neq 2$ \cite{blackpari14}.
Further exposition can be found in Black--Parimala \cite{blackpari14}.

Suffice it to say that Totaro's question has a rich history but is wide open.
In particular, it is completely unexplored in the literature for tori.
Our main result is (cf. Section 5)

\begin{thm}
\label{mainthm}
Totaro's question has a positive answer for tori of rank $\leq 2$.
\end{thm}

We remark that the theorem is independent of the perfection of the ground field.
Define the \emph{separable index} of a variety $X$ over a field, denoted $\ind_\textup{s}(X)$, to be the minimal positive degree of a zero-cycle of closed \etale points on $X$.
The question of equality between $\ind(X)$ and $\ind_\textup{s}(X)$ was raised by Lang--Tate and answered affirmatively by recent work of Gabber--Liu--Lorenzini when $X$ is geometrically smooth, regular, and of finite type over a field \cite{gll13}.
Since torsors under smooth tori over fields satisfy these hypotheses, we need only consider \emph{separable} field extensions in the proof of Theorem \ref{mainthm}.

Now, if $X$ is regular over a field and $U \subseteq X$ is open and dense, then $\ind(X) = \ind(U)$ by a general moving lemma for zero-cycles.
So the index is a birational invariant among regular varieties over a given field.
Together with Theorem \ref{mainthm}, we obtain from this (cf. Section 6)

\begin{cor}
\label{smoocomp}
Let $X$ be a regular variety over a field containing a principal homogeneous space of a smooth torus of rank $\leq 2$ as a dense open subset.
If $X$ admits a zero-cycle of degree $d \geq 1$, then $X$ has a closed \etale point of degree dividing $d$.
\end{cor}

In particular, Manin proved that del Pezzo surfaces of degree 6 are toric varieties as in Corollary \ref{smoocomp} \cite{manin72}.
So as a special case of the corollary, we have

\begin{cor}
Let $X$ be a del Pezzo surface of degree 6.
If $X$ admits a zero-cycle of degree $d \geq 1$, then $X$ has a closed \etale point of degree dividing $d$.
\end{cor}

\noindent
\textbf{Acknowledgments.}
We thank R. Parimala, V. Suresh, N. Bhaskhar, and J.-L. \CT\ for their tremendous insights, advice, and support throughout the writing of this paper.
The author was partially supported by the Robert W. Woodruff Foundation and the National Science Foundation under grants DMS-1401319 and FRG-1463882.

\section{Preliminaries on Tori}

Let $k$ be a field and $k^{\textup{s}}$ be its separable closure. 
For any \etale algebra $A/k$, let $\G_{m,A}$ (or just $\G_m$ when the base is understood) be the abelian group scheme $\Spec A[t,t^{-1}]$. 
A connected linear algebraic group $T/k$ is called an \emph{algebraic torus}, \emph{k-torus}, or simply a \emph{torus} if \[T_{k^{\textup{s}}} := T \times_k k^{\textup{s}} \cong \G_{m,k^{\textup{s}}}^r\] for some $r \geq 1$, which is called the \emph{rank} of the torus. 
If $E/k$ is a field extension such that $T_{E} \cong \G_{m,E}^r$, then $E$ is called a \emph{splitting field} of (and is said to \emph{split}) $T$.

For any finite \etale algebra $A/k$, let $R_{A/k}$ denote the \emph{Weil restriction} functor (also known as the \emph{restriction of scalars} functor), which takes $A$-schemes to $k$-schemes and, in particular, takes $A$-tori to $k$-tori.
In particular, for any finite separable field extension $L/k$ and any $L$-torus $T$, $R_{L/k}T$ is a $k$-torus.
A $k$-torus $T$ is called \emph{quasitrivial} if it is isomorphic to a finite product of tori of the form $R_{L_i/k}\G_m$ where each $L_i/k$ is a finite separable field extension.
For any finite separable field extension $L/k$, call \[R_{L/k}^{(1)}\G_m := \ker [ R_{L/k}\G_m \xrightarrow{N_{L/k}} \G_m ]\] the \emph{norm torus} associated to that extension; $R_{L/k}^{(1)}\G_m$ evidently has rank $[L:k] - 1$.

Now, let $\Gamma = \Gal(k^{\textup{s}}/k)$.
For any rank $r$ $k$-torus $T$, define its \emph{character module} to be \[\chargroup(T) := \Hom(T_{k^{\textup{s}}},\G_{m,k^{\textup{s}}})\ [\ \cong \Hom(\G_{m,k^{\textup{s}}}^r, \G_{m,k^{\textup{s}}}) \cong \Z^r\, ].\] 
Then $\chargroup(T)$ is a rank $r$ $\Gamma$-module.
The association $T \mapsto \chargroup(T)$ is an antiequivalence between the categories of $k$-tori and finitely-generated $\Gamma$-modules; in fact, it is an antiequivalence between the categories of $k$-tori split by a finite Galois extension $E/k$ and finitely-generated $\Gal(E/k)$-modules.
The $\Gamma$-action on $\chargroup(T)$ yields a continuous representation \[\Gamma \to \Aut(\chargroup(T)) \cong \Aut(\Z^r) \cong \GL_r(\Z)\] whose kernel $\mathfrak{h} \ \unlhd \ \Gamma$ corresponds to the minimal splitting field of $T$, a finite Galois extension $E/k$. 
The group $\GL_r(\Z)$ contains the image of this representation, a copy of $\Gamma/\mathfrak{h} \cong \Gal(E/k)$.
Call this the \emph{Galois group of} $T$.
On the other hand, an embedding $\Gal(E/k) \to \GL_r(\Z)$ lifts to a continuous representation $\Gamma \to \GL_r(\Z)$, which determines a $\Gamma$-action on $\chargroup(\G_m^r)$, identifying the rank $r$ $k$-torus $\Spec \left( E[\chargroup(\G_m^r)]^{\Gamma} \right)$ whose Galois group is $\Gal(E/k)$.
Explicitly, \[\begin{array}{rcl}
\{\textup{rank }r\textup{ }k\textup{-tori} \}/{\cong} & \leftrightarrow & \{\textup{rank }r\textup{ }\Gamma\textup{-modules} \}/{\cong} \\
& \leftrightarrow & H^1(k,\Aut(\chargroup(\G_m^r))) \\
& \leftrightarrow & H^1(k,\Aut(\Z^r)) \\
& \leftrightarrow & H^1(k,\GL_r(\Z)) \\
& = & \Hom(\Gamma,\GL_r(\Z))/{\sim}
\end{array} \] where $\rho \sim \rho'$ if and only if $\rho(\Gamma)$ and $\rho'(\Gamma)$ are conjugate in $\GL_r(\Z)$.

To classify rank $r$ tori, it is necessary to count the conjugacy classes of finite subgroups of $\GL_r(\Z)$.
There are 13 such classes in $\GL_2(\Z)$; in \cite{vosk65}, however, Voskresenski{\u\i} gave explicit representations of \emph{15} finite groups in terms of matrix generators along with their associated rank 2 tori.
He later corrected this in a short geometric proof that rank 2 tori are rational \cite{vosk98}; here, he noted that there are only two distinct maximal finite subgroups of $\GL_2(\Z)$ up to conjugacy, $D_4$ and $D_6$, whereas he produced two faithful representations of each of these groups in $\GL_2(\Z)$ in his earlier classification paper.
For the convenience of the cross-referencing reader, the proof of Theorem \ref{mainthm} will follow his original classification.

\section{Lemmata}

In order to prove Theorem \ref{mainthm}, a number of key lemmas will be cited repeatedly.

\begin{lem}
\label{serres}
Totaro's question for $\ind(X) = 1$ has a positive answer for tori.
\end{lem}

\begin{proof}
If $T$ is a torus, then by a well-known fact from Galois cohomology, the composition of the natural restriction and corestriction maps associated to any finite field extension $L/k$ \[ H^1(k,T) \xrightarrow{\textup{res}} H^1(L,T_L) \xrightarrow{\textup{cores}} H^1(k,T) \] is the multiplication-by-$[L:k]$ map.
Now, fix $X \in H^1(k,T)$. 
If $X_L = 0 \in H^1(L,T_L)$ for some finite field extension $L/k$, then \[ [L:k]X = (\textup{cores} \circ \textup{res})(X) = \textup{cores}(0) = 0 \in H^1(k,T),\] and so $\per(X) \mid [L:k]$.
Since $L$ is arbitrary, $\per(X) \mid \ind(X)$.
But $\ind(X) = 1$.
Then $\per(X) = 1$, meaning that $X = 0 \in H^1(k,T)$.
So it suffices to take $F = k$, as desired.
\end{proof}

\begin{lem}
\label{normtori} Let $L/k$ be a finite separable field extension and $T = R_{L/k}^{(1)}\G_m$.
\begin{enumerate}[{\normalfont(a)}]
\item
$H^1(k, T) \cong k^\times/N_{L/k}(L^\times)$. 
\item 
If $L/k$ is cyclic, then $H^1(k, T) \cong \Br(L/k)$.
\item
$H^1(L,T_L) = 0$. In particular, $\ind(X) \mid [L:k]$ for all $X \in H^1(k,T)$.
\end{enumerate}
\end{lem}

\begin{proof}
From the short exact sequence of $k$-tori \[ 1 \to R_{L/k}^{(1)}\G_m \to R_{L/k}\G_m \xrightarrow{N_{L/k}} \G_m \to 1, \] taking Galois cohomology yields the long exact sequence of groups \[ L^\times \xrightarrow{N_{L/k}} k^\times \to H^1(k,T) \to 1,\] from which (a) is clear.
Now, for any finite cyclic field extension $L/k$ with $\Gal(L/k) \cong \langle \sigma \rangle$, \[k^\times/N_{L/k}(L^\times) \cong \Br(L/k)\] via the map \[ \gamma \mapsto (L/k,\sigma,\gamma)\] where $(L/k,\sigma,\gamma)$ is the cyclic algebra generated over $L$ by $u$ with relations $ux = \sigma(x)u$ for any $x \in L$ and $u^{[L:k]} = \gamma$.
From this, (b) follows immediately.
Finally, if $L \cong k[x]/\left( p(x) \right)$ and $a_1,\ldots,a_m$ are the roots of $p(x)$ in $L$, then \[p(x) = q(x)\prod\limits_{i=1}^m (x-a_i) \] for some $q(x) \in L[x]$.
By the Chinese Remainder Theorem, \[ \begin{array}{rcl}
L \otimes_k L & \cong & L \otimes_k k[x]/\left( p(x) \right) \\
& \cong &  L[x]/\left( q(x) \right) \times \prod\limits_{i=1}^m L[x]/\left( x-a_i \right) \\
& \cong & L \times A
\end{array}
\] where $A/L$ is a finite \etale algebra.
So the following diagram commutes.
\[
\begin{tikzpicture}[node distance = 2cm, auto]
\node (C) {$L$};
\node (A) [above of = C, left of = C] {$L \otimes_k L$};
\node (B) [above of = C, right of = C] {$L \times A$};

\draw[->] (A) to node {$\sim$} (B);
\draw[->] (A) to node [swap] {$N_{L \otimes_k L/L}$} (C);
\draw[->] (B) to node {$\id \cdot N_{A/L}$} (C);
\end{tikzpicture}
\]
In particular, $N_{L \otimes_k L/L}$ is surjective since \[ (\id \cdot N_{A/L})(\lambda,1,\ldots,1) = \lambda \] for any $\lambda \in L$.
Then \[ H^1(L,T_L) \cong L^\times/N_{L \otimes_k L/L}\left((L \otimes_k L)^\times\right) = 0,\] hence (c).
\end{proof}

\begin{lem}
\label{minsplit}
Let $T$ be a $k$-torus with a (not necessarily minimal) splitting field $E$ of finite degree over $k$, and let $X \in H^1(k,T)$.
\begin{enumerate}[{\normalfont (a)}]
\item $\ind(X) \mid [E:k]$.
\item If $[E:k]$ is prime, then Totaro's question has an positive answer for $T$.
\end{enumerate}
\end{lem}

\begin{proof}
Since $T_E$ is split, $H^1(E,T_E) = 0$ by Hilbert 90.
Then $\ind(X) \mid [E:k]$.
If $[E:k]$ is prime, then by (a), $\ind(X) = 1$ or $[E:k]$, hence either $F = k$ or $E$ suffices, respectively.
\end{proof}

For any finite extension of \etale algebras $A/B$, let $(A^\times)_{B}^{(1)} := \{a \in A^\times : N_{A/B}(a) = 1 \}$.

\begin{lem}
\label{mntori}
Consider the following diagram of separable field extensions
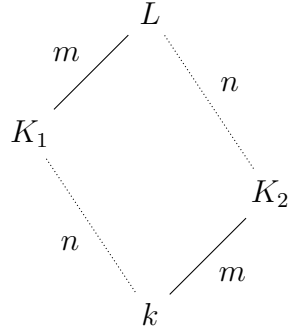
\[
\begin{tikzpicture}[node distance = 0.8cm, auto]
\node (L) {$L$};
\node (K1) [below of = L, below of = L, left of = L, left of = L] {$K_1$};
\node (K2) [below of = L, below of = L, below of = L, right of = L, right of = L] {$K_2$};
\node (k) [below of = L, node distance = 4cm] {$k$};

\draw[-] (L) to node [swap] {$m$} (K1);
\draw[densely dotted] (L) to node {$n$} (K2);
\draw[densely dotted] (k) to node {$n$} (K1);
\draw[-] (k) to node [swap] {$m$} (K2);
\end{tikzpicture}
\]
for some $m,n > 1$, and let $T = R_{K_1/k}\left(R_{L/K_1}^{(1)}\G_m\right) \cap R_{K_2/k}\left(R_{L/K_2}^{(1)}\G_m\right)$.

\begin{enumerate}[{\normalfont (a)}]
\item
The following sequences of $k$-tori are exact.
\[1 \to T \to R_{K_1/k}(R_{L/K_1}^{(1)}\G_m) \xrightarrow{N_{L/K_2}} R_{K_2/k}^{(1)}\G_m \to 1\]
\[1 \to T \to R_{K_2/k}(R_{L/K_2}^{(1)}\G_m) \xrightarrow{N_{L/K_1}} R_{K_1/k}^{(1)}\G_m \to 1\]

\item
The following sequences of abelian groups are exact.
\[(L^\times)_{K_1}^{(1)} \xrightarrow{N_{L/K_2}} (K_2^\times)_{k}^{(1)} \to H^1(k,T) \xrightarrow{\delta_1} K_1^\times/N_{L/K_1}(L^\times)\]
\[(L^\times)_{K_2}^{(1)} \xrightarrow{N_{L/K_1}} (K_1^\times)_{k}^{(1)} \to H^1(k,T) \xrightarrow{\delta_2} K_2^\times/N_{L/K_2}(L^\times)\]
\end{enumerate}
\end{lem}

\begin{proof}
Left exactness of both sequences is clear from the construction of $T$, so proving (a) amounts to showing that $N_{L/K_2}$ and $N_{L/K_1}$ are surjective after extending scalars to $k^{\textup{s}}$.
If $\Phi : (k^{\textup{s}})^{mn} \to (k^{\textup{s}})^n$ and $\Psi : (k^{\textup{s}})^{mn} \to (k^{\textup{s}})^m$ are the maps defined by \[ \Phi(x_{ij} : 1 \leq i \leq m, 1 \leq j \leq n) = \left( \prod\limits_{i = 1}^{m}x_{i1},\ldots, \prod\limits_{i = 1}^{m}x_{in} \right) \] and \[ \Psi(x_{ij} : 1 \leq i \leq m, 1 \leq j \leq n) =   \left( \prod\limits_{j = 1}^{n}x_{1j},\ldots, \prod\limits_{j = 1}^{n}x_{mj} \right),\] then the following diagram commutes.
\[
\begin{tikzpicture}[node distance = 0.65cm, auto]
\node (L) {$L$};
\node (K1) [below of = L, below of = L, left of = L, left of = L, left of = L, left of = L, left of = L, left of = L, left of = L] {$K_1$};
\node (K2) [below of = L, below of = L,below of = L, right of = L, right of = L, right of = L, right of = L, right of = L, right of = L] {$K_2$};
\node (k) [left of = L, below of = L, below of = L, below of = L, below of = L, below of = L] {$k$};

\draw[->] (L) to node [swap] {$N_{L/K_1}$} (K1);
\draw[->] (L) to node {$N_{L/K_2}$} (K2);
\draw[<-] (k) to node {$N_{K_1/k}$} (K1);
\draw[<-] (k) to node [swap] {$N_{K_2/k}$} (K2);

\node (F) [above of = L, above of = L, above of = L, above of = L, above of = L, above of = L] {$(k^{\textup{s}})^{mn}$};
\node (M1) [below of = F, below of = F, left of = F, left of = F, left of = F, left of = F, left of = F, left of = F, left of = F] {$(k^{\textup{s}})^n$};
\node (M2) [below of = F, below of = F, below of = F, right of = F, right of = F, right of = F, right of = F, right of = F, right of = F] {$(k^{\textup{s}})^m$};
\node (l) [left of = F, below of = F, below of = F, below of = F, below of = F, below of = F] {$k^{\textup{s}}$};

\draw[->] (F) to node [swap] {$\Phi$} (M1);
\draw[->] (F) to node {$\Psi$} (M2);
\draw[<-] (l) to node {$N_{(k^\textup{s})^n/k^\textup{s}}$} (M1);
\draw[<-] (l) to node [swap] {$N_{(k^\textup{s})^m/k^\textup{s}}$} (M2);

\draw[->] (L) to node [swap] {$\otimes_k k^{\textup{s}}$} (F);
\draw[->] (K1) to node {$\otimes_k k^{\textup{s}}$} (M1);
\draw[->] (K2) to node [swap] {$\otimes_k k^{\textup{s}}$} (M2);
\draw[->] (k) to node [swap] {$\otimes_k k^{\textup{s}}$} (l);
\end{tikzpicture}
\]
Any $a \in (R_{K_2 \otimes_k k^{\textup{s}}/k^{\textup{s}}}^{(1)}\G_m)(k^{\textup{s}})$ then corresponds to an $m$-tuple $(a_1,\ldots,a_m) \in (k^{\textup{s}})^m$ such that $\prod\limits_{i = 1}^m a_i = 1$.
But $\Psi$ is surjective: if $x_{ij} = a_i$ when $j = 1$ and $x_{ij} = 1$ otherwise, then \[\begin{array}{rcl} \Psi(x_{ij} : 1 \leq i \leq m, 1 \leq j \leq n) & = & \Psi(a_1,\underbrace{1,\ldots,1}_\text{$n-1$ times},a_2,\underbrace{1,\ldots,1}_\text{$n-1$ times},\ldots,a_m,\underbrace{1,\ldots,1}_\text{$n-1$ times}) \\
& = & (a_1,\ldots,a_m),
\end{array}\] and in fact, \[\begin{array}{rcl} \Phi(x_{ij} : 1 \leq i \leq m, 1 \leq j \leq n) & = & \Phi(a_1,\underbrace{1,\ldots,1}_\text{$n-1$ times},a_2,\underbrace{1,\ldots,1}_\text{$n-1$ times},\ldots,a_m,\underbrace{1,\ldots,1}_\text{$n-1$ times}) \\
& = & (\underbrace{1,\ldots,1}_\text{$n$ times}).
\end{array}\] 
So this $mn$-tuple yields a $k^{\textup{s}}$-point of $R_{L \otimes_k k^{\textup{s}}/K_1 \otimes_k k^{\textup{s}}}^{(1)}\G_m$ mapping to $a \in R_{K_2 \otimes_k k^\textup{s}/k^\textup{s}}^{(1)}(k^{\textup{s}})$.
Then $N_{L/K_2}$ is surjective as a map of algebraic groups.
By a symmetric argument, $N_{L/K_1}$ is surjective too, proving (a).
(b) follows by taking Galois cohomology of these short exact sequences of $k$-tori and applying Lemma \ref{normtori}.
\end{proof}

\section{Technical Results} 
Two technical propositions are needed for the proof of Theorem \ref{mainthm}.

\begin{prop}
\label{z4tori}
Let $L/K/k$ be a tower of separable quadratic extensions with no intermediate fields between $k$ and $L$ other than $K$, and let \[T = R_{K/k}(R_{L/K}^{(1)}\G_m).\]
Then Totaro's question has a positive answer for $T$.
\end{prop}

\begin{proof}
Let $M$ be the Galois closure of $L/k$ in $k^{\textup{s}}$ and $G = \Gal(M/k)$.
Either $M = L$, in which case $G \cong \Z/4\Z$, or $[M:L] = 2$, in which case $G \cong D_4$.
Suppose that $M = L$. Then \[ K \otimes_k L \cong L \times L \] as $K \subseteq L$, $[K:k] = 2$, and $K/k$ is separable, and \[ L \otimes_k L \cong (L \times L) \times (L \times L) \] as $[L:k] = 4$ and $L/k$ is Galois.
So the following diagram commutes.
\[
\begin{tikzpicture}[node distance = 2cm, auto]
\node (C) {$K \otimes_{k} L$};
\node (D) [right of = C, right of = C] {$L \times L$};
\node (A) [above of = C] {$L \otimes_k L$};
\node (B) [right of = A, right of = A] {$(L \times L) \times (L \times L)$};

\draw[->] (A) to node {$\sim$} (B);
\draw[->] (A) to node [swap] {$N_{L \otimes_k L/K \otimes_k L}$} (C);
\draw[->] (B) to node {$N_{L \times L/L} \times N_{L \times L/L}$} (D);
\draw[->] (C) to node {$\sim$} (D);
\end{tikzpicture}
\]
Since $N_{L \times L/L} \times N_{L \times L/L}$ is surjective, so is $N_{L \otimes_k L/K \otimes_k L}$, and so by Lemma \ref{normtori}.(a), \[ H^1(L,T_L) \cong (K \otimes_k L)^\times/N_{L \otimes_k L/K \otimes_k L}\left( (L \otimes_k L)^\times \right) = 0.\]
If $[M:L] = 2$, then since $D_4$ contains three distinct subgroups of order 2, there is another tower of separable extensions $M/L'/k$ such that $[M:L'] = 2$, \[K \otimes_k L' \cong M, \] and \[ L \otimes_k L' \cong M \times M.\]
So the following diagram commutes.
\[
\begin{tikzpicture}[node distance = 2cm, auto]
\node (C) {$K \otimes_{k} L'$};
\node (D) [right of = C, right of = C] {$M$};
\node (A) [above of = C] {$L \otimes_k L'$};
\node (B) [right of = A, right of = A] {$M \times M$};

\draw[->] (A) to node {$\sim$} (B);
\draw[->] (A) to node [swap] {$N_{L \otimes_k L'/K \otimes_k L'}$} (C);
\draw[->] (B) to node {$N_{M \times M/M}$} (D);
\draw[->] (C) to node {$\sim$} (D);
\end{tikzpicture}
\]
Since $N_{M \times M/M}$ is surjective, so is $N_{L \otimes_k L'/K \otimes_k L'}$, and so by Lemma \ref{normtori}.(a), \[ H^1(L',T_{L'}) \cong (K \otimes_k L')^\times/N_{L \otimes_k L'/K \otimes_k L'}\left( (L \otimes_k L')^\times \right) = 0.\]
So $\ind(X) \mid 4$ for any $X \in H^1(k,T)$, and if $\ind(X) = 4$, then either $F = L$ or $L'$ will suffice.

Suppose now that $\ind(X) = 2$.
Let $X = [\beta]$ with some $\beta \in K^\times$ that is not a norm from $L^\times$.
Since $\ind(X) = 2$, it can be assumed by Theorem 9.2 from Gabber--Liu--Lorenzini \cite{gll13} using standard Galois theory reductions (cf. Lemma 1.5 from Garibaldi--Hoffman \cite{garihoff06}) that there is a tower of separable field extensions $E'/E/k$ such that $[E':E] = 2$, $[E : k] = m$ for some odd $m$, and \[ \beta \in N_{L \otimes_k E'/K \otimes_k E'}\left((L \otimes_k E')^\times\right).\]
Write \[ 
E' \cong \left\{ \begin{array}{lr} E[x]/(x^2+x+a) & \textup{ if } \ch(k) = 2 \\ 
E[x]/(x^2-a) & \textup{ if } \ch(k) \neq 2 \end{array} \right. 
\] for some $a \in E^\times$.
In both cases, identify the class of $x$ with $i \in E'$.
Then there are $u_0,v_0 \in LE$ not both zero such that \[ 
\begin{array}{rcl}
\beta & = & N_{L \otimes_k E'/K \otimes_k E'}(u_0 + v_0i) \\
& = & \left(N_{LE/KE}(u_0) + aN_{LE/KE}(v_0)\right) + T_{E}(u_0,v_0)i
\end{array}
\] where \[ 
T_K(u,v) = \left\{ \begin{array}{lr} \tr_{L/K}(u\overline{v}) + N_{L/K}(v) & \textup{ if } \ch(k) = 2 \\ 
\tr_{L/K}(u\overline{v}) & \textup{ if } \ch(k) \neq 2 \end{array} \right. 
\]
Since $\beta \in K^\times$, $T_{E}(u_0,v_0) = 0$, and so \[\beta = N_{LE/KE}(u_0) + aN_{LE/KE}(v_0).\]
If $v_0 = 0$, then $\beta = N_{LE/KE}(u_0)$, in which case $\beta \in K^\times$ is represented by the $K$-quadratic form $N_{L/K}$ after extending scalars to $KE$.
But $[KE:K] = [E:k] = m$ is odd.
Then by Springer's Theorem \cite{springer52}, $\beta \in N_{L/K}(L^\times)$, a contradiction.
So $v_0 \neq 0$.

Now, write \[ K \cong \left\{ \begin{array}{lr} k[y]/(y^2+y+b) & \textup{ if } \ch(k) = 2 \\ k[y]/(y^2-b) & \textup{ if } \ch(k) \neq 2 \end{array} \right. \] for some $b \in k^\times$.
In both cases, identify the class of $y$ with $j \in K$.
Then there are $\beta_1, \beta_2 \in k$ not both zero such that \[ \beta = \beta_1 + \beta_2 j .\] 
Let $N^1, N^2 : L \to k$ and $Q^1, Q^2 : L^2 \to k$ be the $k$-quadratic forms defined by \[ N_{L/K} = N^1 + N^2 j, \] \[ Q^1(u,v) = \beta_1 N^1(u)  + b\beta_2 N^2 (u) - N^1 (v), \]
and \[ Q^2(u,v) = \left\{
\begin{array}{lr}
(\beta_1 + \beta_2) N^2 (u) + \beta_2 N^1 (u) + N^2 (v) & \textup{ if } \ch(k) = 2 \\
\beta_1 N^2 (u) + \beta_2 N^1 (u) - N^2 (v) & \textup{ if } \ch(k) \neq 2
\end{array}
\right.
\]
Then setting $x_0 = v_0^{-1}$ and $y_0 = u_0 v_0^{-1}$, \[ 
\begin{array}{rcl}
a & = &  \beta N_{LE/KE}(x_0) - N_{LE/KE}(y_0) \\
& = & (\beta_1 + \beta_2 j) (N_{LE}^1 + N_{LE}^2 j)(x_0) - (N_{LE}^1 + N_{LE}^2 j)(y_0) \\
& = & Q_{E}^1(x_0, y_0) + Q_{E}^2(x_0, y_0)j.
\end{array}
\]
Since $a \in E^\times$, $Q_{E}^1(x_0, y_0) = a$ and $Q_{E}^2(x_0, y_0) = 0$.
Now, case by $\ch(k)$.

First, suppose that $\ch(k) \neq 2$.
Since $\tr_{LE/KE}(y_0) = 0$, the isotropic vector for $Q_E^2$ comes from the subspace \[LE \oplus (LE)^0 \cong (L \oplus L^0) \otimes_k E\] where $L^0 = \ker \tr_{L/K} \subseteq L$.
But as $[E:k] = m$ is odd, $Q^2$ is isotropic by Springer's Theorem \cite{springer52}.
So there is some $(x_1,y_1) \in L \oplus L^0$ such that \[ Q^2(x_1,y_1) = \beta_1 N^2 (x_1) + \beta_2 N^1 (x_1) - N^2(y_1) = 0. \]
If $x_1 = 0$, then $y_1$ is an isotropic vector for $N^2$.
But isotropic quadratic forms are universal.
So for any $x$, there is a $y$ such that $N^2(y) = \beta_1 N^2 (x) + \beta_2 N^1 (x)$, i.e., $Q^2(x,y) = 0.$
Then we can assume that $x_1 \neq 0$.
So \[
\begin{array}{rcl}
\alpha & = & Q^1(x_1,y_1) \\
& = & Q^1(x_1,y_1) + Q^2(x_1,y_1)j \\
& = & \beta N_{L/K}(x_1) - N_{L/K}(y_1)
\end{array}
\] means that \[  N_{L/K}(x_1^{-1})(N_{L/K}(y_1) + \alpha) = \beta.\]
With $F = k(\sqrt{\alpha})$, $[F:k] = 2$, and since $y_1 \in L^0$, \[ N_{L \otimes_k F/K \otimes_k F}\left(\dfrac{y_1 + \sqrt{\alpha}}{x_1}\right) = \beta. \]
Then $X_F = 0 \in H^1(F,T_F)$, as desired.

Now, suppose that $\ch(k) = 2$.
Let $T^1, T^2 : L \to k$ be the $k$-linear maps defined by \[ \tr_{L/K} = T^1 + T^2 j.\]
Since \[
\begin{array}{rcl}
(T_E^1(y_0) + 1) + T_E^2(y_0)j & = & \tr_{LE/KE}(y_0)+1\\ 
& = & \tr_{LE/KE}(u_0v_0^{-1}) + 1 \\
& = & N_{LE/KE}(v_0)\left(\tr_{LE/KE}(u_0\overline{v_0}) + N_{LE/KE}(v_0)\right) \\
& = & 0,
\end{array}
\] $T_E^2(y_0) = 0$, and so the isotropic vector for $Q_E^2$ comes from the subspace \[LE \oplus (LE)^\# \cong (L \oplus L^\#) \otimes_k E\] where $L^\# = \ker T^2 \subseteq L$.
But as $[E:k] = m$ is odd, $Q^2$ is isotropic by Springer's Theorem \cite{springer52}.
So there is some $(x_1,y_1) \in L \oplus L^\#$ such that \[ Q^2(x_1,y_1) = (\beta_1 + \beta_2) N^2 (x_1) + \beta_2 N^1 (x_1) + N^2 (y_1)= 0. \]
If $x_1 = 0$, then $y_1$ is an isotropic vector for $N^2$.
But the symmetric bilinear form \[ b_{N^2} : L^2 \to k \] defined by \[b_{N^2}(x,y) := N^2(x+y) - N^2(x) - N^2(y) = T^2(x\overline{y})\] is non-degenerate.
Then $N^2$ is regular and isotropic, hence universal \cite{ekm08}.
So as before, we can assume that $x_1 \neq 0$.
Let $\gamma = T^1(y_1)$.
If $\gamma = 0$, then $y_1 = 0$ as $y_1 \in L^\#$.
Setting $\alpha = Q^1(x_1, 0)$ and $F = k[z]/(z^2 + z + \alpha)$ and identifying the class of $z$ with $\lambda \in F$ yields that \[N_{L \otimes_k F/K \otimes_k F}\left(\dfrac{\lambda}{x_1}\right) = \beta.\] 
If $\gamma \neq 0$, then \[
N_{L \otimes_k F/K \otimes_k F}\left(\dfrac{y_1 + \gamma\lambda}{\gamma x_1} \right) = \beta.
\]
In both cases, $[F:k] = 2$ and $X_F = 0 \in H^1(F,T_F)$, as desired.
\end{proof}

\begin{prop}
\label{mnprop}
Consider the following diagram of separable field extensions
\[
\begin{tikzpicture}[node distance = 0.8cm, auto]
\node (L) {$L$};
\node (K1) [below of = L, below of = L, left of = L, left of = L] {$K_1$};
\node (K2) [below of = L, below of = L, below of = L, right of = L, right of = L] {$K_2$};
\node (k) [below of = L, node distance = 4cm] {$k$};

\draw[-] (L) to node [swap] {$m$} (K1);
\draw[densely dotted] (L) to node {$n$} (K2);
\draw[densely dotted] (k) to node {$n$} (K1);
\draw[-] (k) to node [swap] {$m$} (K2);
\end{tikzpicture}
\]
for some coprime $m, n > 1$, and let \[T = R_{K_1/k}\left(R_{L/K_1}^{(1)}\G_m\right) \cap R_{K_2/k}\left(R_{L/K_2}^{(1)}\G_m\right).\]
Then Totaro's question has an positive answer for $X \in H^1(k,T)$ of index $m$, $n$, and $mn$.
Furthermore, if $(\ind(X), m) = 1$, then $\ind(X) \mid n$, and if $(\ind(X), n) = 1$, then $\ind(X) \mid m$.
\end{prop}

\begin{proof}
By Lemma \ref{mntori}.(c), the following sequences of abelian groups are exact.
\[(L^\times)_{K_1}^{(1)} \xrightarrow{N_{L/K_2}} (K_2^\times)_{k}^{(1)} \to H^1(k,T) \xrightarrow{\delta_1} K_1^\times/N_{L/K_1}(L^\times)\]
\[(L^\times)_{K_2}^{(1)} \xrightarrow{N_{L/K_1}} (K_1^\times)_{k}^{(1)} \to H^1(k,T) \xrightarrow{\delta_2} K_2^\times/N_{L/K_2}(L^\times)\]
The proof will proceed according to the index.

First, suppose that $\ind(X) = m$.
Since $[L : K_2] = n$, $K_2^\times/N_{L/K_2}(L^\times)$ is $n$-torsion.
But $(m,n) = 1$, and $\per(X) \mid \ind(X)$.
So $\delta_2(X) = 1$.
Then $X$ lifts to some $\beta \in (K_1^\times)_{k}^{(1)}$.
Now, \[ K_2 \otimes_k K_2 \cong K_2 \times B \] where $B/K_2$ is an \etale algebra as $K_2/k$ is separable, \[ K_1 \otimes_k K_2 \cong L \] as $K_1, K_2 \subseteq L$ have coprime degrees and are therefore $k$-linearly disjoint such that \[ [K_1 : k][K_2 : k] = mn = [L : k],\] and \[ L \otimes_k K_2 \cong L \times A \] where $A \cong B \otimes_{K_2} L/L$ is an \etale algebra as $K_2 \subseteq L$ and $K_2/k$ is separable.
After identifying through the natural isomorphisms, the following diagram commutes.
\[
\begin{tikzpicture}[node distance = 0.65cm, auto]
\node (L) {$L$};
\node (K1) [below of = L, below of = L, left of = L, left of = L, left of = L, left of = L, left of = L, left of = L, left of = L] {$K_1$};
\node (K2) [below of = L, below of = L,below of = L, right of = L, right of = L, right of = L, right of = L, right of = L, right of = L] {$K_2$};
\node (k) [left of = L, below of = L, below of = L, below of = L, below of = L, below of = L] {$k$};

\draw[->] (L) to node [swap] {$N_{L/K_1}$} (K1);
\draw[->] (L) to node {$N_{L/K_2}$} (K2);
\draw[<-] (k) to node {$N_{K_1/k}$} (K1);
\draw[<-] (k) to node [swap] {$N_{K_2/k}$} (K2);

\node (F) [above of = L, above of = L, above of = L, above of = L, above of = L, above of = L] {$L \times A$};
\node (M1) [below of = F, below of = F, left of = F, left of = F, left of = F, left of = F, left of = F, left of = F, left of = F] {$L$};
\node (M2) [below of = F, below of = F, below of = F, right of = F, right of = F, right of = F, right of = F, right of = F, right of = F] {$K_2 \times B$};
\node (l) [left of = F, below of = F, below of = F, below of = F, below of = F, below of = F] {$K_2$};

\draw[->] (F) to node [swap] {$\id \cdot N_{A/L}$} (M1);
\draw[->] (F) to node {$N_{L/K_2} \times N_{A/B}$} (M2);
\draw[<-] (l) to node {$N_{L/K_2}$} (M1);
\draw[<-] (l) to node [swap] {$\id \cdot N_{B/K_2}$} (M2);

\draw[->] (L) to node [swap] {$\otimes_k K_2$} (F);
\draw[->] (K1) to node {$\otimes_k K_2$} (M1);
\draw[->] (K2) to node [swap] {$\otimes_k K_2$} (M2);
\draw[->] (k) to node [swap] {$\otimes_k K_2$} (l);
\end{tikzpicture}
\]
Observe that \[(\id \cdot N_{A/L})(\beta,1) = \beta\] and \[(N_{L/K_2} \times N_{A/B})(\beta,1) = \left(N_{K_1/k}(\beta), N_{A/B}(1)\right) = (1,1),\] meaning that $X_{K_2} = 0 \in H^1(K_2, T_{K_2})$.
Since $\ind(X) = [K_2 : k] = m$, it suffices to take $F = K_2$.
But only that $(\ind(X),n) = 1$ is needed to show that $X_{K_2} = 0$.
So $(\ind(X),n) = 1$ implies that $\ind(X) \mid m$.
By a symmetric argument, $F = K_1$ suffices when $\ind(X) = n$, and $(\ind(X),m) = 1$ implies that $\ind(X) \mid n$.

Now, suppose that $\ind(X) = mn$.
Since the sequence of $k$-tori \[ 1 \to T \to R_{K_1/k}(R_{L/K_1}^{(1)}\G_m) \xrightarrow{N_{L/K_2}} R_{K_2/k}^{(1)}\G_m \to 1 \] is short exact, so is the sequence of $K_2$-tori \[ 1 \to T_{K_2} \to R_{L/K_2}(R_{L \times A/L}^{(1)}\G_m) \xrightarrow{N_{L \times A/K_2 \times B}} R_{K_2 \times B/K_2}^{(1)}\G_m \to 1.\]
Since $K_2^{\textup{s}}$-points of $R_{K_2 \times B/K_2}^{(1)}\G_m$ take the form $(N_{B \otimes_{K_2} K_2^{\textup{s}}/K_2^{\textup{s}}}(\beta^{-1}),\beta)$ for $\beta \in (B \otimes_{K_2} K_2^{\textup{s}})^\times$, \[R_{K_2 \times B/K_2}^{(1)}\G_m \cong \G_{m,B}.\] 
By a similar argument, \[R_{L/K_2}(R_{L \times A/L}^{(1)}\G_m) \cong R_{L/K_2}\G_{m,A}.\]
So \[ 1 \to T_{K_2} \to R_{L/K_2}\G_{m,A} \xrightarrow{N_{A/B}} \G_{m,B} \to 1 \] is a short exact sequence of $K_2$-tori.
Since $A/L$ is an \etale algebra, $H^1(L,\G_{m,A}) = 0$ by Hilbert 90.
Taking Galois cohomology then yields the long exact sequence of abelian groups \[ A^\times \xrightarrow{N_{A/B}} B^\times \to H^1(K_2,T_{K_2}) \to 0.\]
So $X_{K_2}$ lifts to some $\beta \in B^\times$.
Let $C/L$ be the \etale algebra such that \[L \otimes_{K_2} L \cong L \times C.\]
Then since \[ 
\begin{array}{rcl}
A \otimes_{K_2} L & \cong & B \otimes_{K_2} L \otimes_{K_2} L \\
& \cong & B \otimes_{K_2} (L \times C) \\
& \cong & A \times (B \otimes_{K_2} C), 
\end{array}\] the following diagram commutes.
\[
\begin{tikzpicture}[node distance = 2cm, auto]
\node (C) {$B \otimes_{K_2} L$};
\node (D) [right of = C, right of = C] {$A$};
\node (A) [above of = C] {$A \otimes_{K_2} L$};
\node (B) [right of = A, right of = A] {$A \times (B \otimes_{K_2} C)$};

\draw[->] (A) to node {$\sim$} (B);
\draw[->] (A) to node [swap] {$N_{A \otimes_{K_2} L/A}$} (C);
\draw[->] (B) to node {$\id \cdot N_{B \otimes_{K_2} C/A}$} (D);
\draw[->] (C) to node {$\sim$} (D);
\end{tikzpicture}
\]
But \[ (\id \cdot N_{B \otimes_{K_2} C/A})(\beta,1) = \beta, \] meaning that $X_L = (X_{K_2})_L = 0 \in H^1(L,T_L)$.
Since $[L:k] = mn$, $F = L$ suffices.
\end{proof}

\begin{cor}
\label{pqcor}
Consider the following diagram of separable field extensions
\[
\begin{tikzpicture}[node distance = 0.8cm, auto]
\node (L) {$L$};
\node (K1) [below of = L, below of = L, left of = L, left of = L] {$K_1$};
\node (K2) [below of = L, below of = L, below of = L, right of = L, right of = L] {$K_2$};
\node (k) [below of = L, node distance = 4cm] {$k$};

\draw[-] (L) to node [swap] {$p$} (K1);
\draw[densely dotted] (L) to node {$q$} (K2);
\draw[densely dotted] (k) to node {$q$} (K1);
\draw[-] (k) to node [swap] {$p$} (K2);
\end{tikzpicture}
\]
for some distinct primes $p$ and $q$, and let \[T = R_{K_1/k}\left(R_{L/K_1}^{(1)}\G_m\right) \cap R_{K_2/k}\left(R_{L/K_2}^{(1)}\G_m\right).\]
Then Totaro's question has a positive answer for $T$.
\end{cor}

\begin{proof}
The claim follows immediately from Proposition \ref{mnprop}.
\end{proof}

\section{Proof of Theorem \ref{mainthm}}

The proof of Theorem \ref{mainthm} will proceed according to $\Gal(E/k)$ where $E$ is the minimal splitting field of the torus.
Recall that for a given group, there may be multiple isomorphism classes of tori associated to that group (over suitably general fields) depending on how many conjugacy classes represent its isomorphism class in $\GL_2(\Z)$.
Finally: by Lemma \ref{serres} and Lemma \ref{minsplit}, one can reduce $\ind(X)$ to be a non-trivial proper divisor of $[E:k]$.

\subsection{Rank 1 Tori}

There are only two (conjugacy classes of) finite subgroups of $\GL_1(\Z) \cong \Z/2\Z$: $(1)$ and $\Z/2\Z$.
These correspond to the two classes of rank 1 tori.
For both types, a positive answer to Totaro's question is a trivial consequence of the previous reductions.

\subsubsection{$\Gal(E/k) \cong (1)$ and $T \cong \G_m$}

\begin{proof}
$T$ is quasitrivial, and so we are done by Hilbert 90.
\end{proof}

\subsubsection{$\Gal(E/k) \cong \mathbb{Z}/2\mathbb{Z}$ and $T \cong R_{E/k}^{(1)}\G_m$}

\begin{proof}
$[E:k]$ is prime, and so we are done by Lemma \ref{minsplit}.(b).
\end{proof}

\subsection{Rank 2 Tori}

There are 9 isomorphism classes and 15 conjugacy classes of finite subgroups of $\GL_2(\Z)$.

\subsubsection{$\Gal(E/k) \cong (1)$ and $T \cong \G_m \times \G_m$}

\begin{proof}
$T$ is quasitrivial, and so we are done by Hilbert 90.
\end{proof}

\subsubsection{$\Gal(E/k) \cong \Z/2\Z$}

\begin{enumerate}[(a)]
\item $T \cong R_{E/k}^{(1)} \times R_{E/k}^{(1)}\G_m$
\item $T \cong \G_m \times R_{E/k}^{(1)}\G_m$
\item $T \cong R_{E/k}\G_m$
\end{enumerate}

\begin{proof}
$[E:k]$ is prime, and so we are done by Lemma \ref{minsplit}.(b).
\end{proof}

\subsubsection{$\Gal(E/k) \cong \Z/2\Z \times \Z/2\Z$}

\begin{center}
\begin{tikzpicture}[node distance = 2cm, auto]

\node (E) {$E$};
\node (L1) [below of = E, left of = E] {$L_1$};
\node (L2) [below of = E, right of = E] {$L_2$};
\node (k) [below of = E, node distance = 4cm] {$k$};

\draw[-] (E) to node [swap] {2} (L1);
\draw[-] (E) to node {2} (L2);
\draw[-] (k) to node {2} (L1);
\draw[-] (k) to node [swap] {2} (L2);
\end{tikzpicture}
\end{center}

\begin{enumerate}[(a)]
\item $T \cong R_{L_1/k}\left(R_{E/L_1}^{(1)}\G_m\right)$

\begin{proof}
Since $[E:k] = 4$, we can assume that $\ind(X) = 2$.
Then \[H^1(k,T) \cong H^1(L_1,R_{E/L_1}^{(1)}\G_m) \cong \Br(E/L_1)\] by Lemma \ref{normtori}.(b).
Let $\delta : H^1(k,T) \to \Br(E/L_1)$ denote the composition.
Since \[
\begin{array}{rcl}
\delta(X_{L_2}) & \cong & \delta(X) \otimes_k L_2 \\
& \cong & \delta(X) \otimes_{L_1} L_1 \otimes_k L_2 \\
& \cong & \delta(X) \otimes_{L_1} E
\end{array}\] is split and $[L_2:k] = 2$, it suffices to take $F = L_2$.
\end{proof}

\item $T \cong R_{L_1/k}^{(1)}\G_m \times R_{L_2/k}^{(1)}\G_m$

\begin{proof}
Since $[E:k] = 4$, we can assume that $\ind(X) = 2$.
As \[
\begin{array}{rcl} 
H^1(k,T) & \cong & H^1(k, R_{L_1/k}^{(1)}\G_m \times R_{L_2/k}^{(1)}\G_m) \\ 
& \cong & H^1(k, R_{L_1/k}^{(1)}\G_m) \times H^1(k,R_{L_2/k}^{(1)}\G_m) 
\\ & \cong & \Br(L_1/k) \times \Br(L_2/k)
\end{array}
\] by Lemma \ref{normtori}.(b), $X$ can be identified with a pair of division algebras $D_1 \in \Br(L_1/k)$ and $D_2 \in \Br(L_2/k)$. 
Since $D_1$ and $D_2$ are both split over quadratic extensions $L_1$ and $L_2$, respectively, each is either a field or a quaternion division algebra. 
If either of $D_1$ or $D_2$ is a field, then it suffices to take either $F = L_2$ or $L_1$, respectively. So we can assume that both $D_1$ and $D_2$ are quaternion division algebras.

Let $D = D_1 \otimes_k D_2$. 
By Albert's Theorem \cite{albert72}, either $D$ is a division algebra or $D_1$ and $D_2$ have a common subfield $F$ separable over $k$ such that $[F:k] = 2$ that necessarily splits both algebras. 
Suppose that $D$ is a division algebra.
Then \[ \ind(D) = \deg(D) = \deg(D_1)\deg(D_2) = 4.\]
But since $\ind(X) = 2$, it can be assumed by Theorem 9.2 from Gabber--Liu--Lorenzini \cite{gll13} using standard Galois theory reductions (cf. Lemma 1.5 from Garibaldi--Hoffman \cite{garihoff06}) that there is a tower of separable field extensions $K'/K/k$ such that $[K':K] = 2$ and $[K:k] = m$. 
Since $[K:k]$ is odd and $\ind(D) = 4$, $D_K$ is a division algebra.
But as $D_{K'}$ is split and $[K':K] = 2$, \[\ind(D) = \ind(D_{K}) = 2,\] a contradiction.
So $D_1$ and $D_2$ have a common subfield $F$ separable over $k$ such that $[F:k] = 2$ that necessarily splits both algebras, completing the proof. 
\end{proof}
\end{enumerate}

\subsubsection{$\Gal(E/k) \cong \Z/3\Z$ and $T \cong R_{E/k}^{(1)}\G_m$}

\begin{proof}
$[E:k]$ is prime, and so we are done by Lemma \ref{minsplit}.(b).
\end{proof}

\subsubsection{$\Gal(E/k) \cong \Z/4\Z = \langle \phi \rangle$ \textup{ and } $T \cong R_{E^{\phi^2}/k}\left(R_{E/E^{\phi^2}}^{(1)}\G_m\right)$}

\begin{proof}
We are done by Proposition \ref{z4tori}.
\end{proof}

\subsubsection{$\Gal(E/k) \cong \Z/3\Z \times \Z/2\Z = \langle \theta \rangle \times \langle \tau \rangle$}

\begin{center}
\begin{tikzpicture}[node distance = 0.8cm, auto]
\node (E) {$E$};
\node (L1) [below of = E, below of = E, left of = E, left of = E] {$E^\tau$};
\node (L2) [below of = E, below of = E, below of = E, right of = E, right of = E] {$E^\theta$};
\node (k) [below of = E, node distance = 4cm] {$k$};

\draw[-] (E) to node [swap] {2} (L1);
\draw[densely dotted] (E) to node {3} (L2);
\draw[densely dotted] (k) to node {3} (L1);
\draw[-] (k) to node [swap] {2} (L2);
\end{tikzpicture}
\end{center}

$T = R_{E^{\tau}/k}\left(R_{E/E^{\tau}}^{(1)}\G_m\right) \cap R_{E^{\theta}/k}\left(R_{E/E^{\theta}}^{(1)}\G_m\right)$

\begin{proof}
We are done by Corollary \ref{pqcor}.
\end{proof}

\subsubsection{$\Gal(E/k) \cong S_3 = \langle \theta \rangle \rtimes \langle \tau \rangle$}

\begin{center}
\begin{tikzpicture}[node distance = 0.8cm, auto]
\node (E) {$E$};
\node (L2) [below of = E, below of = E, left of = E, left of = E] {$E^{\theta\tau}$};
\node (L1) [left of = L2] {$E^{\tau}$};
\node (L3) [right of = L2] {$E^{\theta^2\tau}$};
\node (K1) [below of = E, below of = E, below of = E, right of = E, right of = E] {$E^\theta$};
\node (k) [below of = E, node distance = 4cm] {$k$};

\draw[-] (E) to node [swap] {2} (L1);
\draw[-] (E) to node {} (L2);
\draw[-] (E) to node {} (L3);
\draw[densely dotted] (E) to node {3} (K1);

\draw[densely dotted] (k) to node {3} (L1);
\draw[densely dotted] (k) to node {} (L2);
\draw[densely dotted] (k) to node {} (L3);
\draw[-] (k) to node [swap] {2} (K1);
\end{tikzpicture}
\end{center}

\begin{enumerate}[(a)]
\item $T \cong R_{E^{\tau}/k}^{(1)}\G_m$.
\begin{proof} Since $[E:k] = 6$, the only cases to consider are $\ind(X) = 2$ and 3.
But by Lemma \ref{normtori}.(c), only $\ind(X) = 3$ is possible, and $F = E^{\tau}$ suffices by Lemma \ref{normtori}.(c).
\end{proof}

 \item $T \cong R_{E^{\tau}/k}\left(R_{E/E^{\tau}}^{(1)}\G_m\right) \cap R_{E^{\theta}/k}\left(R_{E/E^{\theta}}^{(1)}\G_m\right)$.
\begin{proof} 
We are done by Corollary \ref{pqcor}.
\end{proof}
\end{enumerate}

\subsubsection{$\Gal(E/k) \cong D_4 \cong \Z/4\Z \rtimes \Z/2\Z = \langle \phi \rangle \rtimes \langle \tau \rangle$}

\begin{center}
\begin{tikzpicture}[node distance = 1.5cm, auto]
\node (E) {$E$};
\node (L3) [below of = E] {$E^{\phi^2}$};
\node (L2) [left of = L3] {$E^{\tau}$};
\node (L1) [left of = L2] {$E^{\phi^2 \tau}$};
\node (L4) [right of = L3] {$E^{\phi^3 \tau}$};
\node (L5) [right of = L4] {$E^{\phi\tau}$};
\node (K2) [below of = L3] {$E^{\phi}$};
\node (K1) [left of = K2] {$E^{\phi^2,\tau}$};
\node (K3) [right of = K2] {$E^{\phi^2, \phi\tau}$};
\node (k) [below of = K2] {$k$};

\draw[-] (E) to node [swap] {2} (L1);
\draw[-] (E) to node {} (L2);
\draw[-] (E) to node {} (L3);
\draw[-] (E) to node {} (L4);
\draw[-] (E) to node {2} (L5);
\draw[-] (K1) to node {2} (L1);
\draw[-] (K1) to node {} (L2);
\draw[-] (K1) to node {} (L3);
\draw[-] (K2) to node {} (L3);
\draw[-] (K3) to node {} (L3);
\draw[-] (K3) to node {} (L4);
\draw[-] (K3) to node [swap] {2} (L5);
\draw[-] (k) to node {2} (K1);
\draw[-] (k) to node {} (K2);
\draw[-] (k) to node [swap] {2} (K3);
\end{tikzpicture}
\end{center}

\begin{enumerate}[(a)]
\item $T \cong R_{E^{\phi^2,\tau}/k}\left( R_{E^{\tau}/E^{\phi^2,\tau}}^{(1)}\G_m \right)$

\begin{proof}
We are done by Proposition \ref{z4tori}.
\end{proof}

\item $T \cong R_{E^{\phi^2,\phi\tau}/k}\left( R_{E^{\phi\tau}/E^{\phi^2,\phi\tau}}^{(1)}\G_m \right)$

\begin{proof}
$T$ is isomorphic to the torus from (a).
\end{proof}

\end{enumerate}

\subsubsection{$\Gal(E/k) \cong D_{6} \cong \Z/6\Z \rtimes \Z/2\Z = \langle \sigma \rangle \rtimes \langle \tau \rangle$}

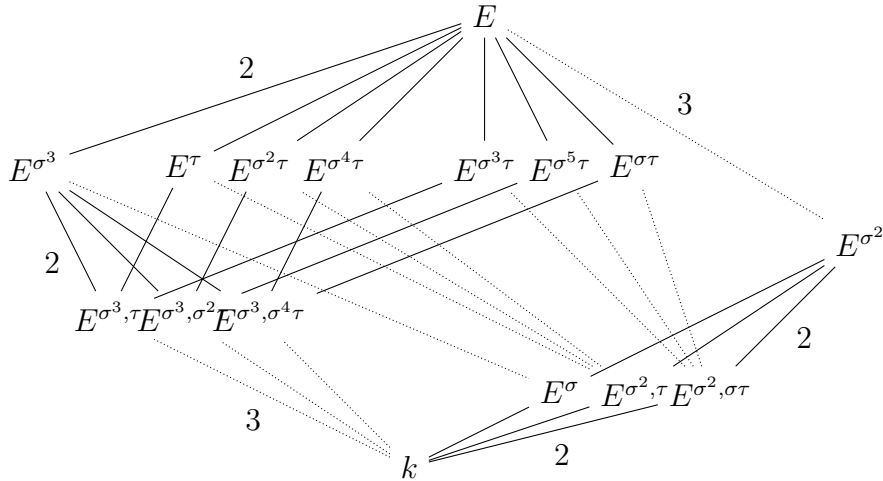
\begin{center}
\begin{tikzpicture}[node distance = 1cm, auto]
\node (E) {$E$};

\node (F4) [below of = E, below of = E, left of = E, left of = E] {$E^{\sigma^{4}\tau}$};
\node (F3) [left of = F4] {$E^{\sigma^{2}\tau}$};
\node (F2) [left of = F3] {$E^{\tau}$};
\node (b11) [left of = F2] {};
\node (F1) [left of = b11] {$E^{\sigma^3}$};
\node (F5) [below of = E, below of = E] {$E^{\sigma^{3}\tau}$};
\node (F6) [right of = F5] {$E^{\sigma^{5}\tau}$};
\node (F7) [right of = F6] {$E^{\sigma\tau}$};

\node (F8) [below of = F7, right of = F7, right of = F7, right of = F7] {$E^{\sigma^2}$};

\node (b21) [below of = b11] {};
\node (b22) [below of = F7, right of = F7] {};

\node (L1) [below of = b21] {$E^{\sigma^3, \tau}$};
\node (L2) [right of = L1] {$E^{\sigma^3, \sigma^2\tau}$};
\node (L3) [right of = L2] {$E^{\sigma^3, \sigma^4\tau}$};
\node (b31) [below of = b22] {};

\node (K3) [below of = b31] {$E^{\sigma^2, \sigma\tau}$};
\node (K2) [left of = K3] {$E^{\sigma^2, \tau}$};
\node (K1) [left of = K2] {$E^{\sigma}$};
\node (b42) [left of = K1] {};
\node (b41) [left of = b42] {};

\node (k) [below of = b41] {$k$};

\draw[-] (E) to node [swap] {2} (F1);
\draw[-] (E) to node {} (F2);
\draw[-] (E) to node {} (F3);
\draw[-] (E) to node {} (F4);
\draw[-] (E) to node {} (F5);
\draw[-] (E) to node {} (F6);
\draw[-] (E) to node {} (F7);

\draw[densely dotted] (E) to node {3} (F8);

\draw[-] (F1) to node [swap] {2} (L1);
\draw[-] (F1) to node {} (L2);
\draw[-] (F1) to node {} (L3);
\draw[densely dotted] (F1) to node {} (K1);

\draw[-] (F2) to node {} (L1);
\draw[-] (F3) to node {} (L2);
\draw[-] (F4) to node {} (L3);

\draw[-] (F5) to node {} (L1);
\draw[-] (F6) to node {} (L2);
\draw[-] (F7) to node {} (L3);

\draw[-] (K1) to node {} (F8);

\draw[densely dotted] (K2) to node {} (F2);
\draw[densely dotted] (K2) to node {} (F3);
\draw[densely dotted] (K2) to node {} (F4);
\draw[-] (K2) to node {} (F8);

\draw[densely dotted] (K3) to node {} (F5);
\draw[densely dotted] (K3) to node {} (F6);
\draw[densely dotted] (K3) to node {} (F7);
\draw[-] (K3) to node [swap] {2} (F8);

\draw[-] (k) to node {} (K1);
\draw[-] (k) to node {} (K2);
\draw[-] (k) to node [swap] {2} (K3);
\draw[densely dotted] (k) to node {3} (L1);
\draw[densely dotted] (k) to node {} (L2);
\draw[densely dotted] (k) to node {} (L3);
\end{tikzpicture}
\end{center}

\begin{enumerate}[(a)]
\item $T \cong R_{E^{\sigma^2}/k}\left(R_{E/E^{\sigma^2}}^{(1)}\G_m\right) \cap R_{E^{\sigma^3}/k}\left(R_{E/E^{\sigma^3}}^{(1)}\G_m\right) \cap R_{E^{\tau}/k}\G_m$

\begin{proof}
Observe that $t \in T(A)$ for a $k$-algebra $A$ if and only if \[ 
\begin{array}{rcl}
t^{\sigma^2}t^{\sigma^4}t & = & 1 \\
t^{\sigma^3}t & = & 1\\
t^\tau & = & t
\end{array}\]
which means that \[ T \cong R_{E^{\sigma^2,\tau}/k}\left(R_{E^{\tau}/E^{\sigma^2,\tau}}^{(1)}\G_m\right) \cap R_{E^{\sigma^3,\tau}/k}\left(R_{E^{\tau}/E^{\sigma^3,\tau}}^{(1)}\G_m\right). \]
So we are done by Proposition \ref{z4tori}.
\end{proof}

\item $T \cong R_{E^{\sigma^2}/k}\left(R_{E/E^{\sigma^2}}^{(1)}\G_m\right) \cap R_{E^{\sigma^3}/k}\left(R_{E/E^{\sigma^3}}^{(1)}\G_m\right) \cap R_{E^{\tau}/k}\left(R_{E/E^{\tau}}^{(1)}\G_m\right)$

\begin{proof}
$T$ is isomorphic to the torus from (a).
\end{proof}
\end{enumerate}

\noindent This exhausts \Vosk's classification and thus completes the proof of Theorem \ref{mainthm}.

\section{del Pezzo Surfaces}

We now prove a general consequence of Theorem \ref{mainthm}.

\begin{cor}
\label{predpcor}
Let $X$ be a regular variety over a field containing a principal homogeneous space of a smooth torus of rank $\leq 2$ as a dense open subset.
If $X$ admits a zero-cycle of degree $d \geq 1$, then $X$ has a closed \etale point of degree dividing $d$.
\end{cor}

\begin{proof}
Write $X = \overline{Y}$ for some principal homogeneous space $Y$ under a torus $T$ of rank $\leq 2$.
By a general moving lemma for zero-cycles (cf. Theorem 6.8 from Gabber--Liu--Lorenzini \cite{gll13}), given a closed point on $X$ of degree $n$, there is a zero-cycle on $Y$ of degree $n$.
So given a zero-cycle on $X$ of degree $d$, there is a zero-cycle on $Y$ of degree $d$.
By Theorem \ref{mainthm}, $Y \subseteq X$ has a closed \etale point of degree dividing $d$.
\end{proof}

A \emph{del Pezzo surface} is a smooth projective surface $X$ over a field $k$ whose anticanonical bundle $\omega_X^{-1}$ is ample.
Its \emph{degree} is the self-intersection number $D = (K_X, K_X)$ of its canonical divisor $K_X$ and lies between 1 and 9.
If $D = 8$, then $X_{k^\textup{s}}$ is isomorphic to either $\mathbb{P}_{k^\textup{s}}^2$ blown up at a point or $\mathbb{P}_{k^\textup{s}}^1 \times \mathbb{P}_{k^\textup{s}}^1$; otherwise, $X_{k^\textup{s}}$ is isomorphic to $\mathbb{P}_{k^\textup{s}}^2$ blown up at $9 - D$ points in general position.
Manin \cite{manin86} is a standard reference for these results; in fact, it is a theorem of Manin that del Pezzo surfaces of degree 6 contain torsors of rank 2 tori as dense open subsets (cf. Teorema 8.6 from \cite{manin72}, Theorem 30.3.1 from \cite{manin86}).
This gives

\begin{cor}
\label{dpcor}
Let $X$ be a del Pezzo surface of degree 6.
If $X$ admits a zero-cycle of degree $d \geq 1$, then $X$ has a closed \etale point of degree dividing $d$.
\end{cor}

\begin{proof}
This follows immediately from Corollary 6.1.
\end{proof}

Of independent interest are the particular rank 2 tori that arise from del Pezzo surfaces of degree 6 within \Vosk's classification.
By the explicit algebraic computations of Blunk \cite{blunk10}, over a non-separably-closed field $k$, each such torus takes the form \[  T = R_{K_2/k}\left(R_{L/K_2}^{(1)}\G_m\right)/R_{K_1/k}^{(1)}\G_m \] for some diagram of separable field extensions
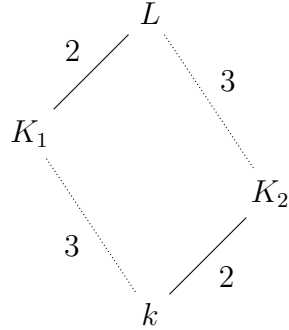
\[
\begin{tikzpicture}[node distance = 0.8cm, auto]
\node (L) {$L$};
\node (K1) [below of = L, below of = L, left of = L, left of = L] {$K_1$};
\node (K2) [below of = L, below of = L, below of = L, right of = L, right of = L] {$K_2$};
\node (k) [below of = L, node distance = 4cm] {$k$};

\draw[-] (L) to node [swap] {$2$} (K1);
\draw[densely dotted] (L) to node {$3$} (K2);
\draw[densely dotted] (k) to node {$3$} (K1);
\draw[-] (k) to node [swap] {$2$} (K2);
\end{tikzpicture}
\]

\begin{lem}
\label{dplemma}
$T \cong R_{K_1/k}\left(R_{L/K_1}^{(1)}\G_m\right) \cap R_{K_2/k}\left(R_{L/K_2}^{(1)}\G_m\right)$.
\end{lem}

\begin{proof}
Let $\Gal(L/K_1) \cong \Z/2\Z = \langle \sigma \rangle$ and \[S = R_{K_1/k}\left(R_{L/K_1}^{(1)}\G_m\right) \cap R_{K_2/k}\left(R_{L/K_2}^{(1)}\G_m\right).\]
It suffices to show that the sequence of $k$-tori \[ 1 \to R_{K_1/k}^{(1)}\G_m \xrightarrow{\iota} R_{K_2/k}\left(R_{L/K_2}^{(1)}\G_m\right) \xrightarrow{\varphi} S \to 1 \] where $\iota$ is the inclusion map and $\varphi$ is defined functorially for any $k$-algebra $A$ by \[
\begin{array}{rcl}
R_{K_2/k}\left(R_{L/K_2}^{(1)}\G_m\right)(A) & \xrightarrow{\varphi(A)} & S(A) \\
a & \mapsto & \sigma(a)a^{-1}
\end{array}\] is short exact.
Left exactness is clear since $K_1 = L^\sigma$, so all that remains is to show that $\varphi$ is surjective after passing to the separable closure $k^{\textup{s}}$.
Let $\beta \in S(k^{\textup{s}})$.
Then \[N_{L \otimes_k k^{\textup{s}}/K_1 \otimes_k k^{\textup{s}}}(\beta) = 1 = N_{L \otimes_k k^{\textup{s}}/K_2 \otimes_k k^{\textup{s}}}(\beta).\]
By Hilbert 90, $\beta = \sigma(\gamma)\gamma^{-1}$ for some $\gamma \in (L \otimes_k k^{\textup{s}})^\times$.
Set $\lambda = N_{L \otimes_k k^{\textup{s}}/K_2 \otimes_k k^{\textup{s}}}(\gamma)$.
Then \[\sigma(\lambda)\lambda^{-1} = N_{L \otimes_k k^{\textup{s}}/K_2 \otimes_k k^{\textup{s}}}(\beta) = 1,\] i.e., $\lambda \in \left((K_2 \otimes_k k^{\textup{s}})^\sigma\right)^\times = (k^{\textup{s}})^\times$.
Since $K_1/k$ is separable and $k^{\textup{s}}$ is separably closed, $K_1 \otimes_k k^{\textup{s}} \cong (k^\textup{s})^3$.
So there is some $\eta \in (K_1 \otimes_k k^{\textup{s}})^\times$ such that $\lambda = N_{K_1 \otimes_k k^{\textup{s}}/k^{\textup{s}}}(\eta)$.
Set $\alpha = \eta^{-1}\gamma$.
Then \[ N_{L \otimes_k k^{\textup{s}}/K_2 \otimes_k k^{\textup{s}}}(\alpha) = N_{L \otimes_k k^{\textup{s}}/K_2 \otimes_k k^{\textup{s}}}\left(\eta^{-1}\gamma\right) = \lambda^{-1}N_{L \otimes_k k^{\textup{s}}/K_2 \otimes_k k^{\textup{s}}}(\gamma) = 1, \] i.e., $\alpha \in R_{K_2/k}\left(R_{L/K_2}^{(1)}\G_m\right)(k^{\textup{s}})$, and \[ \varphi(\alpha) = \varphi\left( \eta^{-1}\gamma\right) = \sigma\left( \eta^{-1}\gamma\right)\left( \eta^{-1}\gamma\right)^{-1} = \sigma(\gamma)\gamma^{-1} = \beta,\] 
completing the proof.
\end{proof}

\section{Conclusions and an Interesting Open Question}

\begin{thm}
Totaro's question has a positive answer for:
\begin{enumerate}[{\normalfont I.}]
\item quasitrivial tori.
\item norm tori of cyclic field extensions.
\item norm tori of prime degree field extensions.
\item tori of rank $r \leq 2$.
\item tori of the form $R_{K_1/k}\left(R_{L/K_1}^{(1)}\G_m\right) \cap R_{K_2/k}\left(R_{L/K_2}^{(1)}\G_m\right)$ where 
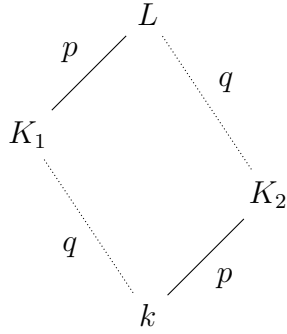
\[
\begin{tikzpicture}[node distance = 0.8cm, auto]
\node (L) {$L$};
\node (K1) [below of = L, below of = L, left of = L, left of = L] {$K_1$};
\node (K2) [below of = L, below of = L, below of = L, right of = L, right of = L] {$K_2$};
\node (k) [below of = L, node distance = 4cm] {$k$};

\draw[-] (L) to node [swap] {$p$} (K1);
\draw[densely dotted] (L) to node {$q$} (K2);
\draw[densely dotted] (k) to node {$q$} (K1);
\draw[-] (k) to node [swap] {$p$} (K2);
\end{tikzpicture}
\] is a diagram of field extensions for distinct primes $p$ and $q$.
\end{enumerate}
\end{thm}

Now, consider the following natural question about division algebras.
\\\\
\noindent
\textbf{Open Question}:
Let $p$ be an odd prime and $D$ and $D'$ be non-split cyclic division algebras over $k$.
If $D_K$ and $D'_K$ share a subfield of degree $p$ over $K$ for some finite separable field extension $K/k$ such that $p \nmid [K:k]$, then do $D$ and $D'$ share a subfield of degree $p$ over $k$?
\\

A negative answer would yield the first known counterexample to Totaro's question.

Let $k$ be a field, and let $L$ and $L'$ be cyclic field extensions of $k$ of degree $p$ such that $D \cong (L/k,\sigma,\gamma)$ and $D' \cong (L'/k,\sigma',\gamma')$.
If $T = R_{L/k}^{(1)}\G_m \times R_{L'/k}^{(1)}\G_m$, then by Lemma \ref{normtori}.(b), $H^1(k,T) \cong \Br(L/k) \times \Br(L'/k)$.
The pair $(D,D')$ then identifies some $X \in H^1(k,T)$ that has a point over $LL'$.
If $L = L'$, then this is the desired common subfield.
Otherwise, $[LL':k] = p^2$.
The condition that $D_K$ and $D'_K$ have a common subfield, say $E$, of degree $p$ over $K$ means that $D_E$ and $D'_E$ are split, and so $X$ has a point over $E$.
But $[E:k] = p[K:k]$.
So $\ind(X) = p$ since $\ind(X) \neq 1$ (because $D$ and $D'$ are non-split) and $\ind(X) \mid (p^2, p[K:k])$.
Since a minimal splitting field of a division algebra is isomorphic to a maximal subfield of the algebra, the open question amounts to Totaro's question for $T$ in the $\ind(X) = p$ case.

As a consequence of our much deeper understanding of quaternion algebras compared to cyclic algebras of odd prime degree, we know that the question has a positive answer when $p = 2$; this is just 5.2.3.(b) in the proof of Theorem \ref{mainthm}.
But unlike in our proof, even having an ``Albert's Theorem'' \cite{albert72} for odd primes would not be strong enough to immediately settle the question because $D \not\cong D^{\textup{opp}}$, and so statements about the splitting fields of $D \otimes_k D'$ seem to be of limited utility.
All this is to say that Totaro's question for tori thinly disguises many fundamental questions about division algebras whose answers, for now, remain elusive.

\bibliographystyle{alpha}
\bibliography{totaro}

\def\cftil#1{\ifmmode\setbox7\hbox{$\accent"5E#1$}\else
  \setbox7\hbox{\accent"5E#1}\penalty 10000\relax\fi\raise 1\ht7
  \hbox{\lower1.15ex\hbox to 1\wd7{\hss\accent"7E\hss}}\penalty 10000
  \hskip-1\wd7\penalty 10000\box7}
\begin{thebibliography}{EKM08}

\bibitem[Alb72]{albert72}
A.~A. Albert.
\newblock Tensor products of quaternion algebras.
\newblock {\em Proc. Amer. Math. Soc.}, 35:65--66, 1972.

\bibitem[BFL90]{bayerlenstra90}
E.~Bayer-Fluckiger and H.~W. Lenstra, Jr.
\newblock Forms in odd degree extensions and self-dual normal bases.
\newblock {\em Amer. J. Math.}, 112(3):359--373, 1990.

\bibitem[Bla11a]{black11a}
J.~Black.
\newblock Implications of the {H}asse principle for zero cycles of degree one
  on principal homogeneous spaces.
\newblock {\em Proc. Amer. Math. Soc.}, 139(12):4163--4171, 2011.

\bibitem[Bla11b]{black11b}
J.~Black.
\newblock Zero cycles of degree one on principal homogeneous spaces.
\newblock {\em J. Algebra}, 334:232--246, 2011.

\bibitem[Blu10]{blunk10}
M.~Blunk.
\newblock Del {P}ezzo surfaces of degree 6 over an arbitrary field.
\newblock {\em J. Algebra}, 323(1):42--58, 2010.

\bibitem[BP14]{blackpari14}
J.~Black and R.~Parimala.
\newblock Totaro's question for simply connected groups of low rank.
\newblock {\em Pacific J. Math.}, 269(2):257--267, 2014.

\bibitem[CTC79]{colliocoray79}
J.-L. Colliot-Th{\'e}l{\`e}ne and D.~Coray.
\newblock L'\'equivalence rationnelle sur les points ferm\'es des surfaces
  rationnelles fibr\'ees en coniques.
\newblock {\em Compositio Math.}, 39(3):301--332, 1979.

\bibitem[EKM08]{ekm08}
R.~Elman, N.~Karpenko, and A.~Merkurjev.
\newblock {\em The algebraic and geometric theory of quadratic forms},
  volume~56 of {\em American Mathematical Society Colloquium Publications}.
\newblock American Mathematical Society, Providence, RI, 2008.

\bibitem[Flo04]{florence04}
M.~Florence.
\newblock Z\'ero-cycles de degr\'e un sur les espaces homog\`enes.
\newblock {\em Int. Math. Res. Not.}, (54):2897--2914, 2004.

\bibitem[GH06]{garihoff06}
S.~Garibaldi and D.~W. Hoffmann.
\newblock Totaro's question on zero-cycles on {$G_2$}, {$F_4$} and {$E_6$}
  torsors.
\newblock {\em J. London Math. Soc. (2)}, 73(2):325--338, 2006.

\bibitem[GLL13]{gll13}
O.~Gabber, Q.~Liu, and D.~Lorenzini.
\newblock The index of an algebraic variety.
\newblock {\em Invent. Math.}, 192(3):567--626, 2013.

\bibitem[Man72]{manin72}
Yu.~I. Manin.
\newblock {\em Kubicheskie formy: algebra, geometriya, arifmetika}.
\newblock Izdat. ``Nauka'', Moscow, 1972.

\bibitem[Man86]{manin86}
Yu.~I. Manin.
\newblock {\em Cubic forms}, volume~4 of {\em North-Holland Mathematical
  Library}.
\newblock North-Holland Publishing Co., Amsterdam, second edition, 1986.
\newblock Algebra, geometry, arithmetic, Translated from the Russian by M.
  Hazewinkel.

\bibitem[Par05]{pari05}
R.~Parimala.
\newblock Homogeneous varieties---zero-cycles of degree one versus rational
  points.
\newblock {\em Asian J. Math.}, 9(2):251--256, 2005.

\bibitem[PR94]{platrap92}
V.~Platonov and A.~Rapinchuk.
\newblock {\em Algebraic groups and number theory}, volume 139 of {\em Pure and
  Applied Mathematics}.
\newblock Academic Press, Inc., Boston, MA, 1994.
\newblock Translated from the 1991 Russian original by Rachel Rowen.

\bibitem[San81]{sansuc81}
J.-J. Sansuc.
\newblock Groupe de {B}rauer et arithm\'etique des groupes alg\'ebriques
  lin\'eaires sur un corps de nombres.
\newblock {\em J. reine angew. Math.}, 327:12--80, 1981.

\bibitem[Ser95]{serre95}
J.-P. Serre.
\newblock Cohomologie galoisienne: progr\`es et probl\`emes.
\newblock {\em Ast\'erisque}, (227):Exp.\ No.\ 783, 4, 229--257, 1995.
\newblock S{\'e}minaire Bourbaki, Vol. 1993/94.

\bibitem[Spr52]{springer52}
T.~A. Springer.
\newblock Sur les formes quadratiques d'indice z\'ero.
\newblock {\em C. R. Acad. Sci. Paris}, 234:1517--1519, 1952.

\bibitem[Tot04]{totaro04}
B.~Totaro.
\newblock Splitting fields for {$E_8$}-torsors.
\newblock {\em Duke Math. J.}, 121(3):425--455, 2004.

\bibitem[Vos65]{vosk65}
V.~E. Voskresenski{\u\i}.
\newblock On two-dimensional algebraic tori.
\newblock {\em Izv. Akad. Nauk SSSR Ser. Mat.}, 29:239--244, 1965.

\bibitem[Vos98]{vosk98}
V.~E. Voskresenski{\u\i}.
\newblock {\em Algebraic groups and their birational invariants}, volume 179 of
  {\em Translations of Mathematical Monographs}.
\newblock American Mathematical Society, Providence, RI, 1998.
\newblock Translated from the Russian manuscript by Boris Kunyavski [Boris
  {\`E}. Kunyavski{\u\i}].

\end{thebibliography}

\end{document}